\documentclass{amsart}
\usepackage{amssymb,amsmath,amsthm,mathrsfs} 
\usepackage[backend=biber, style=alphabetic, doi=false, url=false, isbn=false, giveninits=true, maxbibnames=99]{biblatex}
\usepackage{graphicx} 
\usepackage{verbatim}
\usepackage{tikz-cd} 
\usepackage{tikz}
\usepackage{hyperref}
\usepackage{tabularx}
\usepackage{geometry}
\usepackage{mathtools}
\usepackage{float}
\usepackage{booktabs}
\usepackage{lipsum}

\usepackage{enumerate}

\let\phi\varphi

\newcommand\tab[1][.5cm]{\hspace*{#1}}
\newtheorem{theorem}{Theorem}[section]
\newtheorem{corollary}[theorem]{Corollary}
\newtheorem{lemma}[theorem]{Lemma}
\newtheorem{proposition}[theorem]{Proposition}

\newtheorem{conjecture}[theorem]{Conjecture}
\theoremstyle{remark}
\newtheorem{remark}[theorem]{Remark}

\newcommand*{\bbb}[1]{{\mathbb{#1}}}

\newcommand{\Q}{\bbb{Q}}

\newcommand{\Z}{\bbb{Z}}

\newcommand{\A}{\bbb{A}}

\newcommand{\Ker}{\text{Ker }}
\newcommand{\Gal}{\text{Gal}}

\newcommand{\Cl}{\text{Cl}}

\newcommand{\Nrm}{\text{N}}

\newcommand{\Mod}[1]{\ (\mathrm{mod}\ #1)}
\newcommand{\p}{\mathfrak{p}}

\newcommand{\PP}{\mathfrak{P}}
\newcommand{\bPP}{\bar{\mathfrak{P}}}

\newcommand{\rank}{\text{Rank}}
\newcommand{\Am}{\text{Am}}

\newcommand{\vep}{\varepsilon}

\newcommand{\mO}{\mathcal{O}}
\newcommand{\mP}{\mathcal{P}}

\title{An Analogue of Greenberg's Conjecture for CM Fields}

\author{Peikai Qi}
\email{qipeikai@msu.edu}
\address{Michigan State University, East Lansing, Michigan, USA}
\date{\today}

\author{Matt Stokes}
\email{mathewsonstokes@gmail.com}
\address{Randolph College, Lynchburg, Virginia, USA}
\thanks{Thanks to Lawrence C. Washington for his interest in the topics of this paper.  
The authors would also like to thank Jie Yang and Preston Wake for their useful comments and suggestions.}
\begin{document}

\begin{abstract}
Let $K$ be a CM field and $K^+$ be the maximal totally real subfield of $K$. Assume that the primes above $p$ in $K^+$ split in $K$. Let $S$ be a set containing exactly half of the prime ideals in $K$ above $p$.  We show, assuming Leopoldt's conjecture is true for $K$ and $p$, that there is a unique $\Z_p$-extension of $K$ unramified outside of $S$ (the $S$-ramified $\Z_p$-extension of $K$). Such $\Z_p$-extensions for CM fields have similar properties to the cyclotomic $\Z_p$-extensions of a totally real field. For example, Greenberg proved some criterion for the Iwasawa invariants $\mu=\lambda=0$ of the cyclotomic $\Z_p$-extension of a totally real field, and we will prove analogous results for the $S$-ramified $\Z_p$-extension of a CM field. We also give a numerical criterion for the Iwasawa invariants $\mu=\lambda=0$ for an imaginary biquadratic field, which is analogous to the one given by Fukuda and Komatsu for real quadratic fields.

\end{abstract}
\maketitle
\section{Introduction}
Let $F$ be any number field and $p$ be an odd prime. Let $F\subset F_1\subset F_2\subset\cdots\subset F_n\subset \cdots\subset F_\infty=\bigcup_n F_n$ be a $\Z_p$-extension of $F$, that is, $\Gal(F_n/F)=\Z/p^n\Z$ and $\Gal(F_\infty/F)=\Z_p$. Let $A_n$ be the $p$ primary part of the class group of $F_n$.  Iwasawa \cite{MR0349627} proved that there are three integers $\lambda$, $\mu$, and $\nu$ such that 
\[|A_n|=p^{\mu p^n+\lambda n+\nu}\]
when $n$ is sufficiently large.  These are the so called Iwasawa invariants for the $\Z_p$-extension $F_{\infty}/F$.

For each number field $F$, there is one obvious $\Z_p$-extension. Let $\zeta_{p^n}$ be the primitive $p^n$-th root of unity. Then $\Q(\zeta_{p^{n+1}})/\Q$ has a unique degree $p^n$ sub-extension of $\Q$ denoted as $\Q_n$. Put $\Q_\infty=\bigcup_n \Q_n$ and $F_\infty=F\Q_\infty$. Then $F_\infty/F$ is a $\Z_p$-extension of $F$ which is called the cyclotomic $\Z_p$-extension of $F$. Let $F_n$ be the $n$-th layer of the $\Z_p$-extension of $F_\infty/F$.

When $F$ is a totally real field, Greenberg \cite{MR0401702} conjectured that $\mu=\lambda=0$ for the cyclotomic $\Z_p$-extension of $F$. In other words, the order of the groups $A_n$ should be bounded as $n\rightarrow \infty$ for the cyclotomic $\Z_p$-extension of a totally real field.  Much research has been done to study this well-known conjecture, for example, \cite{MR0845905}\cite{MR1373702}\cite{MR1854114}.  

\subsection{Main Results}
Now, instead of considering the cyclotomic $\Z_p$-extension of a totally real field, we will consider a certain $\Z_p$-extension of a CM field. Throughout the paper $p$ will be an odd prime.  Let $K$ be a CM field and $K^+$ be the maximal subfield fixed by complex conjugation. Let $S^+ = \{\mP_1,\mP_2,\cdots,\mP_s\}$ be a set containing primes of $K^+$ above $p$, and assume that $\mP_i$ splits in $K$ as $\mP_i\mO_K=\PP_i\tilde{\PP}_i$ for $1\leq i\leq s$, where $\tilde{\PP}_i$ is the complex conjugation of $\PP_i$.  We write
\[
S = \{\PP_1, \PP_2, \dots, \PP_s\}.
\]
In Section \ref{sec1}, we prove the following theorem (which is Theorem \ref{unique Zp} in the Section \ref{sec1}).

\begin{theorem}\label{thrm1.1}
  Assume that primes above $p$ in $K^+$ split in the CM field $K$. Then there is a $\Z_p$-extension $K_\infty/K$ unramified outside $S$. If the Leopoldt's conjecture holds, then such $\Z_p$-extension is unique. 
\end{theorem}

We will refer to the $\Z_p$-extension $K_{\infty}/K$ in Theorem \ref{thrm1.1} as the $S$-ramified $\Z_p$-extension of $K$.  Theorem \ref{thrm1.1} is a generalization of Section 2 of Goto \cite{MR2293501}.  The difference in this paper is that we do not assume $p$ splits completely in $K$, nor do we assume $K$ is abelian.  The theorem can also be viewed as an analogue of the fact that there is only one $\Z_p$-extension of a totally real field.

The $S$-ramified $\Z_p$-extensions of a CM field $K$ is similar to the cyclotomic $\Z_p$-extenion of $K^+$ in certain cases.  For instance Fukuda and Komatsu \cite{MR1969639},\cite{MR3240810} give numerical evidence that the $\lambda$-invariant vanishes for $S$-ramified $\Z_p$-extensions of imaginary quadratic fields. Hence, we propose the following analogue of Greenberg's conjecture. 

\begin{conjecture}\label{conjecture}
 Let $K_\infty/K$ be the $S$-ramified $\Z_p$-extension defined as Theorem \ref{unique Zp}. Then $A_n$ will be bounded as $n\rightarrow \infty$. In other words, the Iwasawa invariant $\mu=\lambda=0$. 
\end{conjecture}

While we can not prove this conjecture, we can show that the conjecture holds under some assumptions. From now on, we assume that $p$ is an odd prime and primes above $p$ in $K^+$ split in $K$ and all primes in $S$ are totally ramified in the $S$-ramified $\Z_p$-extension.  They correspond  to the Assumption \eqref{ass1}, \eqref{ass3}, and \eqref{ass2} in Section \ref{sec1}.

Let $i_{n,m}:\Cl(K_n) \rightarrow \Cl(K_m)$ for $m\geq n$ be the natural map between class groups induced by inclusion. Let $H_n=\cup_{m\geq n} \Ker(i_{n,m})$. 
\begin{theorem}
    Assume that $p$ is inert in $K^+/\Q$ and Leopoldt's conjecture holds for $K$. Then the following are equivalent.
    \begin{enumerate}
        \item $A_0=H_0$.
        \item $|A_n|$ is bounded as $n\rightarrow \infty$.
    \end{enumerate}  
\end{theorem}
This is Theorem \ref{p inert} in Section \ref{sec2}. Let $B_n$ be the subgroup of $A_n$ fixed by $\Gal(K_n/K)$. Let $D_n$ be the subgroup of $A_n$ generated by prime ideals above $S$. 
\begin{theorem}
    Assume $p$ splits completely in $K^+/\Q$ and Leopoldt's conjecture holds for $K$. Then the following are equivalent.
      \begin{enumerate}
        \item $B_n=D_n$ for all sufficiently large $n$.
        \item $|A_n|$ is bounded as $n\rightarrow \infty$.
    \end{enumerate}
\end{theorem}
This is Theorem \ref{TH4.1} in Section \ref{sec3}. In fact, the above two theorems are analogous to Greenberg's results \cite{MR0401702} for the cyclotomic $\Z_p$-extension of a totally real field, and so we follow Greenberg's proofs. 

Let $E(L):=\mO_L^*$ be the group of units of $\mO_L$ for a number field $L$. Let $K$ be a CM field. We next compare $S$-ramified $\Z_p$-extension $K_\infty/K$ and cyclotomic $\Z_p$-extension $K^{+}_\infty/K^+$. Let $K_n$ be the $n$th layer of $K_\infty/K$ and $K_n^+$ be the $n$th layer of $K^{+}_\infty/K^+$. Let $\Nrm_{n}$ be the norm map from field $K_n$ to $K$ or $K_n^+$ to $K^+$.
\begin{proposition}
Assume $p$ splits completely in $K/\Q$ and Leopoldt's conjecture holds for $K$. 
\[
 E(K)/( \Nrm_n(K_n^*) \cap E(K))  \cong  E(K^+) /(\Nrm_n((K^+_n)^*) \cap E(K^+) )
 \]    
\end{proposition}
This is Proposition \ref{P5.5} in Section \ref{sec4}. The proposition is interesting because the extensions $K_n/K$ and $K_n^+/K^+$ are globally unrelated, but locally similar.

There are various kinds of numerical criterion to determine when $\mu=\lambda=0$ for the cyclotomic $\Z_p$-extension of a real quadratic field. We give a similar numerical criterion to that of Fukuda and Komatsu \cite{MR0845905} for the $S$-ramified $\Z_p$-extension of an imaginary biquadratic field.
Let $m,d \in \Z^+$ that are squarefree and coprime.  Denote $k = \Q(\sqrt{-m})$, $F = \Q(\sqrt{d})$, $K = Fk$, and $\vep$ to be the fundamental unit for $K$.  Suppose that $p$ splits completely in $K$, with $p\mO_k = \p\tilde{\p}$ and $\p\mO_{K} = \PP\bPP$.  Take $S=\{\PP,\bPP\}$.  The following is Theorem \ref{TH5.7} in Section \ref{sec4}.
\begin{theorem}
 Suppose $p$ doesn't divide the class number of $K$, and that $r$ is the smallest positive integer such that 
 \[
 \vep^{p-1}\equiv 1 \Mod{\bPP^r}.
 \]
 If $\Nrm_{r-1}(E_{r-1}) = E_0$, then $\mu=\lambda = 0$ for the $S$-ramified $\Z_p$-extension of $K_\infty/K$.
 \end{theorem}

The structure of this paper is as follows: In Section \ref{sec1}, we prove Theorem \ref{unique Zp}. In Section \ref{general property}, we prove some properties for $S$-ramified $\Z_p$-extension and introduce some assumptions. In Section \ref{sec2}, we deal with the case when $p$ is inert in $K^+/\Q$. In Section \ref{sec3}, we deal with the case when $p$ splits completely in $K^+/\Q$. In Section \ref{sec4}, we compare ambiguous class group between $S$-ramified $\Z_p$-extension of $K$ and cyclotomic $\Z_p$-extension of $K^+$. We also give a numerical criterion for $\mu=\lambda=0$ for biquadratic fields at the end of the section.

\subsection{Potential Directions}
Here we remark on some possible future directions one might take in studying Conjecture \ref{conjecture}:

\begin{enumerate}[(i)]
    \item There are many other criterion developed to study Greenberg's conjecture, in particular, criterion relating to real quadratic fields \cite{MR1373702}\cite{MR0845905}. One may try to generalize these criterion to study Conjecture \ref{conjecture} in the biquadratic field case.  
    
 \item  Fukuda and Komatsu  \cite{MR1969639} \cite{MR3240810} only calculate the examples for the $S$-ramified $\Z_p$-extension of $K_\infty/K$ defined in Theorem \ref{unique Zp} for imaginary quadratic fields $K$.  All calculated examples have $\lambda=0$. Is it possible to calculate more examples for the general CM field $K$?

 \item  In \cite{MR0880470} and \cite{MR0871166}, the authors proved that $\mu=0$ for such $S$-ramified $\Z_p$-extension of $K_\infty/K$ when $K$ is imaginary quadratic field. Ferrero-Washington \cite{MR0528968} proved that $\mu=0$ for the cyclotomic $\Z_p$-extension of an abelian number field. Could one adapt their argument to prove that $\mu=0$ for $S$-ramified $\Z_p$-extension when the CM field $K$ is abelian?

\end{enumerate}

\subsection{Another Potential Analogy}

Finally, we present another similarity between the $\Z_p$-extensions of totally real fields, and the $S$-ramified $\Z_p$-extensions of an imaginary quadratic field. Let $F$ be a totally real field and $F_{\infty}/F$ be the cyclotomic $\Z_p$-extension of $F$. Let $F_n$ be the $n$-th layer of $F_\infty/F$. Let $M$ be the maximal pro-$p$ abelian extension of $F$ unramified outside primes above $p$. We know that $\Gal(M/F_\infty)$  is a finite group if Leopoldt's conjecture holds. 

We say $a\sim b$ if two numbers $a,b$ have the same $p$ adic valuation. The following theorem can be found in the appendix of \cite{MR0460282}:
\begin{theorem}[Coates \cite{MR0460282}]\label{coates}
  Under the assumption of Leopoldt’s conjecture, we have
  \[
  \#\Gal(M/F_\infty)\sim
\frac{w_1(F(\mu_p))h_FR_p(F)\prod_{\mP\mid p}(1-(\Nrm\mP)^{-1})}{\sqrt{\Delta_{F/\Q}}}
\]
Here $\mu_p$ is the group of $p$th roots of unity, $w_1(F(\mu_p))$ is the number of roots of unity of $F(\mu_p)$, $h_F$ is the class number of $F$, $R_p(F)$ is the $p$-adic regulator of $F$, $\Nrm\mP$ is the absolute norm of $\mP$, and  $\Delta_{F/\Q}$ is the discriminant of $F$.  
\end{theorem}

Now, let $F$ be a totally real field with $[F:\Q] = d$, $k$ an imaginary quadratic field, and $p\geq 3$ a prime that splits as $p\mO_k = \p\tilde{\p}$.  Assume $p$ doesn't divide the class number of $k$. Let $k \subseteq k_1 \subseteq \dots \subseteq k_{\infty}$ be the unique non-cyclotomic $\Z_p$-extension of $k$ unramified outside $\p$. Denote $K = kF$ and $K\cap k_\infty=k_e$.  Define  $K_n = k_{n+e}F$, $K_\infty=k_\infty F$. Hence, $K_n$ is the $n$-th layer of the $\Z_p$-extension $K_\infty/K$. We know that $K_\infty/K$ is the $S$-ramified $\Z_p$-extension, where $S$ is the set of prime above $\p$ in $K$.

Let $M$ be the maximal  pro-$p$ abelian extension of $K$ unramified outside $S$. 
\begin{theorem}[Coates and Wiles \cite{MR0450241}]\label{coates and wiles}
   Under the assumption of Leopoldt's conjecture, we have
  \[
  \#\Gal(M/K_\infty)\sim \frac{p^{e+1}h_KR_p(K)\prod_{\PP\mid\p}(1-(\Nrm\PP)^{-1})}{\nu_K\sqrt{\Delta_{K/k}}}
  \]  
  Here $e$ is an integer defined by $ K\cap k_\infty=k_e$, $h_K$ is the class number of $K$, $R_p(K)$ is the $p$-adic regulator of $K$, $\Nrm\PP$ is the absolute norm of $\PP$, $\nu_K$ is the order of the group of $p$ power root of unity in $K$ and $\Delta_{K/k}$ is the relative discriminant of $K$ over $k$.
\end{theorem}

\begin{remark}
    The value on the right-hand side of Theorems \ref{coates} and \ref{coates and wiles} can be interpreted as a $p$-adic residue for the $p$-adic $L$ function corresponding to the characteristic polynomial of Iwasawa module $X_\infty$. See the appendix of \cite{MR0460282}.
\end{remark}
\begin{remark}
  Coates and Wiles proved Theorem \ref{coates and wiles} with a different motivation. Their method of proof was also used in their paper on the conjecture of Birch and Swinnerton-Dyer \cite{MR0463176}.  
\end{remark}

One may wonder if an analogous result to Theorem \ref{coates and wiles} for the $S$-ramified $\Z_p$-extension of a CM field exists.

\section{Uniqueness of $S$-Ramified $\Z_p$-Extensions}\label{sec1}
Let $K$ be a CM field and $K^+$ be the maximal subfield fixed by complex conjugation. Let $S^+ = \{\mP_1,\mP_2,\cdots,\mP_s\}$ be the set of primes of $K^+$ above $p$. Assume that each of the primes above $p$ in $K^+$ split in $K$. Write $\mP_i\mO_K=\PP_i\tilde{\PP}_i$ for $1\leq i\leq s$, where $\tilde{\PP}_i$ is the complex conjugation of $\PP_i$, and set $S = \{\PP_1, \PP_2, \dots, \PP_s\}$. The following theorem can be viewed as the analogue of the fact that there is a unique $\Z_p$-extension of a totally real field for which Leopoldt's conjecture holds. The method of the proof is similar to the proof of Theorem 13.4 in Washington \cite{MR1421575}.

\begin{theorem}\label{unique Zp}
   Let $T$ be the maximum abelian extension of $K$ unramified outside $S$.  Then there is a surjective homomorphism $\Gal(T/K) \to \Z_p^{1 + \delta}$ with finite kernel, where $\delta$ is the Leopoldt defect (see \cite[Theorem 13.4 page 266]{MR1421575}).  In particular, if Leopolt's conjecture holds for $K$, then $\delta = 0$ and there is a unique $\Z_p$-extension contained in $T$.
\end{theorem}

\begin{proof}
   Let $T $ be the maximal abelian extension of $K$ which is unramified outside $S$. Let $\A_K^*$ be the group of id\`eles of $K$.  By class field theory, there is a closed subgroup $R\subset \A_K^*$ such that 
\[\Gal(T/K)\cong \A_K^*/R\]

Let $U_{v}$ be the local unit group at a place $v$ of $K$ and $U_{v}=K_v^*$ if $v$ is an archimedean place. Define

\[
U'=\prod_{i=1}^{s}U_{\PP_i}, \tab  U''=\prod_{v\notin S}U_v, \tab U=U'\times U''
\]
We can view $U'$ as a subgroup of $\A_K^*$ by placing a 1 in each component outside of $S$, and we can view $U''$ as a subgroup of $\A_K^*$ in a similar way. Let $W=\overline{K^*U''}$ be the closure of $K^*U''$ inside $\A_K^*$. Since $T$ is unramified outside $S$,  we have $ W \subset  R$. Since $T$ is maximal, we must have $W=R$. Thus $
\Gal(T/K)\cong\A_K^*/W $. 
Let $H$ be the Hilbert class field of $K$. Then a similar argument shows
\[
\Gal(H/K)=\A_K^*/(K^*U).
\]
Therefore, we have $\Gal(T/H)=K^*U/W=U'W/W\cong U'/(U'\cap W)$. Let $U_{1,\PP_i}$ be the group of local units congruent to 1 modulo $\PP_i$, and put $U_1=\prod_{i=1}^{s}U_{1,\PP_i}$. Then 
\[
U'=U_1\times \text{(finite group)}
\]
Hence
\[
\Gal(T/H)/\text{(finite group)}\cong U_1(U'\cap W)/(U'\cap W)=U_1/(U_1\cap W)
\]
Let $E_1$ be the group of units in $K$ congruent to 1 modulo the primes in $S$. Then we can embed $E_1$ in $\A_K^*$ by
\[
\phi: E_1\hookrightarrow U_1\subset \A_K^*.
\]
In a moment we will prove Lemma \ref{E_1 intersection}, which implies
\[
U_1\cap W=U_1\cap \overline{K^*U''}=\overline{\phi(E_1)}.
\]
Writing $E_1(K^+)$ to be the group of units in $K^+$ congruent to 1 modulo the primes in $S^+$, we have 
\[
E_1(K^+)\subset E_1\subset\mO_K^*.
\]

Since $\rank_{\Z}\mO_K^*=\rank_{\Z} E_1(K^+)=[K:\Q]/2-1$, the index of $E_1(K^+)$ in $E_1$ is finite. Hence, the index of $\overline{E_1(K^+)}$ in $\overline{E_1}$ is finite. Assume that $ \rank_{\Z_p}(\overline{E_1(K^+)})=[K:\Q]/2-1-\delta$ and $\delta\geq 0$. Then
\[
\overline{\phi(E_1)}\cong \Z_p^{[K:\Q]/2-1-\delta} \times \text{(finite group)}.
\]
Recall that \cite[Page 75]{MR1421575} Leopoldt's conjecture predicts that $\delta=0$.
By \cite[Prop 5.7]{MR1421575}, we know
\[
U_1\cong \text{(finite group)}\times \Z_p^{[K:\Q]/2}.
\]

Therefore
\[
U_1/(U_1\cap W)=\text{(finite group)} \times \Z_p^{1+\delta}.
\]
Hence
\[
\Gal(T/H)=\text{(finite group)} \times \Z_p^{1+\delta}.
\]
Since $\Gal(H/K)\cong \Cl(K)$ is a finite group, 
\[
\Gal(T/K)/\Z_p^{1+\delta}\cong\text{(finite group)}.
\]
We will also prove Lemma \ref{quotient finite}, which shows there is a finite group such that 
\[
\Gal(T/K)/\text{(finite group)}\cong \Z_p^{1+\delta}.
\]
  Let $\tilde{K}$ be the compositum of all $\Z_p$-extension of $K$ unramified outside $S$. The fixed field of this finite group must be $\Tilde{K}$ and $\Gal(\Tilde{K}/K)\cong \Z_p^{1+\delta}$. If Leoplodt's conjecture holds for $K$, then Leoplodt's conjecture holds for $K^+$. Hence $\delta=0$, which implies there is a unique $\Z_p$-extension of $K$ unramified outside $S$.

\end{proof}

\begin{lemma}\label{E_1 intersection}
   $U_1\cap W=U_1\cap \overline{K^*U''}=\overline{\phi(E_1)} $
\end{lemma}
\begin{proof}
Take $\varepsilon\in E_1$.  Then $\phi(\varepsilon)\in U_1$ and 
\[
\phi(\varepsilon)=\varepsilon \frac{\phi(\varepsilon)}{\varepsilon}.
\]
We have $\varepsilon\in K^*$ and $\phi(\varepsilon)/\varepsilon\in U''$. Hence $\overline{\phi(E_1)}\subset U_1\cap W$.

Recall that in a topological space, the closure generated by a set $V$ is the intersection of closed subsets containing $V$. Define $U_{n,\PP_i}=\{x\in U_{\PP_i}|x\equiv 1 \mod \PP_i^n\}$ and  $U_n=\prod_{i=1}^s U_{n,\PP_i}$. Then
\[
W=\overline{K^*U''}=\bigcap_n K^*U''U_n
\]
\[
\phi(E_1)=\bigcap_n \phi(E_1)U_n 
\]
It suffices to show that 
\[
U_1\cap K^*U''U_n\subset \phi(E_1)U_n 
\]
Take $x\in K^*$, $u''\in U''$ and $u \in U_n $ such that $xu''u\in U_1$. Then $ xu''\in U_1$. At primes of $S$, we have  $u''=1$, so $x$ is a local unit. At the primes outside of $S$, we have that $u''$ is a unit, hence $x$ is also a local unit. Therefore, $x$ is a local unit everywhere which implies $x$ is a global unit. At the primes of $S$ we have $xu''=x$, and at the primes outside of $S$ we have $xu''=1$. This exactly means that $\phi(x)=xu''$. Hence 
\[
xu''u\in \phi(E_1)U_n.
\]
This completes the proof of the lemma.
\end{proof}

\begin{lemma}\label{quotient finite}
    Let $J$ be a profinite abelian group and $ \Z_p^r\subset J$. Assume that $J/\Z_p^r$ is a finite group.  Then there exists a finite group $T$ such that 
      \[
    J/T \cong \Z_p^r
    \]
\end{lemma}

\begin{proof}
 Write $G = J/\Z_p^r$ and denote $N$ to be the size of $G$. Then
\[
N\Z_p^r\subset NJ\subset \Z_p^r
\]
Hence $NJ\cong \Z_p^r$. Putting $J[N]=\{x\in J|Nx=0\}$, we have
\begin{equation*}
    \begin{split}
        J/J[N]&\cong  NJ\cong \Z_p^r\\
            x& \rightarrow Nx
    \end{split}
\end{equation*}
In fact, $\#J[N]\leq N$ is finite by the snake lemma.
    \[
    \begin{tikzcd}
       0\arrow[r] &0\arrow[d] \arrow[r]&J[N] \arrow[r]\arrow[d]& G[N]\arrow[d]& \\
       0\arrow[r]&\Z_p^r\arrow[r] \arrow[d,"N"]&J\arrow[r] \arrow[d,"N"]&
       G\arrow[r] \arrow[d,"N"]&0\\       0\arrow[r]&\Z_p^r\arrow[r] &J\arrow[r] &
       G\arrow[r] &0
    \end{tikzcd}
    \]
\end{proof}

\section{Analogue of Greenberg's Criterion For $S$-Ramified $\Z_p$-Extensions of CM Fields}\label{general property}

Let $K$ be a CM field. Throughout the paper, we will assume that
\begin{equation}\label{ass1}
    \text{each of the primes above } p \text{ in } K^+ \text{ split in } K
\end{equation} 
\begin{equation}\label{ass3}
    p \text{ is an odd prime.}
\end{equation}

We will consider the  $\Z_p$-extension $K_\infty/K$ unramified outside $S$ which exists by Theorem \ref{unique Zp}. It is unique if Leopoldt's conjecture holds for $K$. We call this $\Z_p$-extension of the CM field $K$ the $S$-ramified $\Z_p$-extension.

We will show that such $S$-ramified $\Z_p$-extensions $K_\infty/K$  for CM field $K$ have similar properties as cyclotomic $\Z_p$-extensions of totally real fields. First, we prove a series of results analogous to that of Greenberg's in \cite{MR0401702}, the proofs of which closely follow the arguments found in there.

In the case of the cyclotomic $\Z_p$-extension of totally real fields, any prime above $p$ will become totally ramified starting from some higher layer. However, we do not know whether the primes inside $S$ are ramified in the $S$-ramified $\Z_p$ extension $K_\infty/K$. We make the following assumption for the whole paper. 
\begin{equation}\label{ass4}
    \text{For any } S\text{-ramified } \Z_p\text{-extension } K_\infty/K, \text{ assume that all primes in } S \text{ are ramified in } K_\infty/K
\end{equation}

Let $K_n$ be the $n$th layer of the $S$-ramified $\Z_p$-extension $K_\infty/K$. Let $A_n$ denote the $p$-primary part of the class group of $K_n$, and let $\sigma$ be a topological generator of $\Gal(K_\infty/K)\cong \Z_p$.  Recall that for each prime $\mP_i$ of $K^+$ we assume $\mP_i\mO_K = \PP_i\tilde{\PP}_i$ and denote $S = \{\PP_1, \PP_2, \dots, \PP_s\}$. Following the notation of Greenberg, we define
\[
B_n:=\{c\in A_n\mid c^\sigma =c\}=A_n^{\sigma}.
\]
Greenberg \cite[Proposition 1]{MR0401702} showed that $|B_n|$ is bounded for the cyclotomic $\Z_p$-extension of a totally real field assuming Leopoldt's conjecture. We have the following similar result.

\begin{proposition}\label{B bounded}
   Suppose that Leopoldt's conjecture holds for $p$ and $K$. Then $|B_n|$ is bounded as $n\rightarrow \infty$ for the unique $S$-ramified $\Z_p$-extension of $K$. 
\end{proposition}

\begin{proof}
    Let $T$ be the maximal abelian extension of $K$ unramified outside of $S$.  By Theorem \ref{unique Zp} we have that $\Gal(T/K_{\infty}) < \infty$. Let $L_n'$ be the maximal pro-$p$ abelian extension of $K_n$ that is unramified over $K_n$. Then $A_n \cong \Gal(L_n'/K_n)$. Let $L_n$ be the maximal pro-$p$ abelian extension of $K$ that is unramified over $K_n$, so $A_n^{\sigma-1}\cong \Gal(L_n'/L_n)$. 
    Then $[L_n:K_n]=[A_n:A_n^{\sigma-1}]=|B_n|$. When $n$ is large enough, $K_\infty/K_n$ is totally ramified. Thus $L_n\cap K_\infty=K_n$. Hence when $n$ is large enough, we have $|B_n|=[L_n:K_n]=[L_nK_\infty:K_\infty]\leq [T:K_\infty]\leq \infty$.
\end{proof}

Let $ D_n$ be the subgroup of $A_n$ which consists of ideal classes containing a product of prime ideals above the primes of $S$.

    \begin{remark}
        In \cite{MR0401702}, Greenberg defines $ D_n$ to be the subgroup of $A_n$ which consists of ideal classes that contain a product of prime ideals above $p$. In that situation, all primes above $p$ are ramified in the cyclotomic $\Z_p$-extension. The difference in our case is that only the primes in $S$ are ramified in the $S$-ramified $\Z_p$-extension.
    \end{remark}

    Denote $e$ to be the smallest positive integer such that the primes of $K_e$ above $S$ are totally ramified in $K_{\infty}/K_e$.
    
\begin{corollary} \label{D not in B}
Let $[\alpha]\in D_0$.  Then the ideal $\alpha$ will become principal in $K_m$ when $m$ is sufficiently large.
\end{corollary}
\begin{proof}
  Take $\PP_i\in S$. Let $\Omega_{i,n}$ be the product of prime above $\PP_i$ in $K_n$. Then $\Omega_{i,n}\mO_{K_m}=\Omega_{i,m}^{p^{m-n}}$ when $m\geq n\geq e$ since $K_m/K_n$ is totally ramified at primes above $S$. Assume $\alpha= \prod_i\PP_i^{s_i}$ for some integer $s_i$. Then $\alpha \mO_{K_n}$ is a product of $\Omega_{i,n}$. We have $\alpha \mO_{K_m}=b_m^{p^{m-n}}$ for some ideal $b_m$ in $K_m$ and $b_m$ is a product of $\Omega_{i,m}$. Hence, $b_m\in B_m$. By Proposition \ref{B bounded}, we know $B_m$ is bounded. Hence $\alpha \mO_{K_m}$ become principal when $m$ is large enough. 
\end{proof}
The same argument can show the following result when $K_\infty/K$ is totally ramified at all primes in $S$. 
\begin{corollary} \label{D become principal}
Let $[\alpha]\in D_n$ and $e=0$.  Then the ideal $\alpha$ will become principal in $K_m$ when $m\geq n$ is large enough.
\end{corollary}

\begin{proof}
    Let $\PP_{i,n}$ be the prime of $K_n$ that lies above $\PP_{i} \in S$. Then $\PP_{i,n}\mO_{K_m}=\PP_{i,m}^{p^{m-n}}$.  We may assume that $\alpha= \prod_i\PP_{i,n}^{s_i}$ for some integers $s_i$. Then $\alpha\mO_{K_m}=\prod_i\PP_{i,n}^{s_i}\mO_{K_m}=\prod_i\PP_{i,m}^{{s_i}p^{(m-n)}}$ for $m\geq n$. On the other hand, the class of $\prod_i\PP_{i,m}^{s_i}$ is contained in $B_m$, and the size $B_m$ is bounded as $m\rightarrow \infty$. Hence $\alpha\mO_{K_m}$ becomes principal when $m$ is large enough. 
\end{proof}

Let $i_{n,m}:\Cl(K_n)\rightarrow \Cl(K_m)$ for $m>n$ be the map induced by inclusion, let $\Nrm_{m,n}:\Cl(K_m)\rightarrow \Cl(K_n) $ for $m>n$ be the norm map, and write $H_{n,m}=\Ker(i_{n,m})$. Because $\Nrm_{m,n}\circ i_{n,m}=p^{m-n}$, we have that $H_{n,m}\subset A_{n}$. Put $H_n=\bigcup_{m\geq n}H_{n,m}$. Corollary \ref{D not in B} says that $D_0\in H_0$, and Corollary \ref{D become principal} says that $D_n\subset H_n$ when $e=0$. Greenberg \cite[Proposition 2]{MR0401702} proved that the size of $A_n$ is bounded as $n \to \infty$ for the cyclotomic $\Z_p$-extension of a totally real field if and only if $H_n=A_n$ for all $n$. We have the following similar result.
 
\begin{proposition} \label{Hn=An}
   Assuming Leopoldt's conjecture holds for $p$ and $K$, we have that $|A_n|$ is bounded if and only if $H_n=A_n$ for all $n$ for the $S$-ramified $\Z_p$-extension.
\end{proposition}
\begin{proof}
       First, Iwasawa proved that $|H_n|$ is bounded \cite[Theorem 10 on page 264]{MR0349627}, so $|A_n|$ is bounded if $A_n = H_n$ for all $n$.

       Now, assume $|A_n|$ is bounded.  Let $L_n'$ be the maximal pro-$p$ abelian extension of $K_n$ that is unramified over $K_n$. Then $A_n \cong \Gal(L_n'/K_n)$. When $m\geq n\geq e$, the field extension  $K_m/K_n$ is totally ramified at primes above $S$. This implies that $K_m\cap L_n'=K_n$, and therefore the restriction map from $\Gal(L_m'/K_m)\rightarrow\Gal(L_n'/K_n)$ is surjective. Hence, the norm map $\Nrm_{m,n}:A_m\rightarrow A_n$ is surjective for $m\geq n\geq e$. Since we are assuming $|A_n|$ is bounded as $n\rightarrow \infty$, we have $\Nrm_{m,n}$ is an isomorphism when $m\geq n\geq n_0$ for some sufficiently large integer $n_0$.

       Take $c\in A_n$. Define $c_r:=i_{n,r}(c)$ where $r\geq n_0$. Take $m$ large enough such that $c_r^{p^{m-r}}=1$. Then $\Nrm_{m,r}(i_{r,m}(c_r))=c_r^{p^{m-r}}=1$. We know $\Ker\Nrm_{m,r}=1$ since $m\geq r\geq n_0$, which implies that $i_{n,m}(c)=i_{r,m}(c_r)=1$. Hence $c\in H_n$ by definition.     
\end{proof}

To simplify our upcoming arguments, we make one more assumption througout the rest of the paper:
\begin{equation}\label{ass2}
    \text{Assume that } K_\infty/K \text{ is totally ramified at all primes in } S. 
\end{equation}
In other words, we further assume $e=0$. Assumption \eqref{ass2} is equivalent to $e=0$ and assumption \eqref{ass4}. One big advantage of assumption \eqref{ass2} is that it implies $D_n\subset B_n$.

\section{The Case Where $p$ is Inert in $K^+$}\label{sec2}

In this section, we assume that $p$ remains prime in $K^+$.  Hence there are only two primes $\PP$ and $\Tilde{\PP}$ above $p$ in $K$, and $S = \{\PP\}$.  We still keep assumptions \eqref{ass1}, \eqref{ass3}, and \eqref{ass2}.

\begin{theorem}\label{p inert}
   Suppose that the odd prime $p$ is inert in $K^+/\Q$ and Leopoldt's conjecture holds. Let $K_\infty/K$ be the unique $S$-ramified $\Z_p$-extension. With the same notation as before, the following statements are equivalent:
    \begin{enumerate}
        \item $A_0=H_0$.
        \item $|A_n|$ is bounded as $n\rightarrow \infty$.
    \end{enumerate}
\end{theorem}
\begin{remark}
    Greenberg's Theorem 1 in \cite{MR0401702} states a similar criterion for cyclotomic $\Z_p$-extension of a totally real field when $p$ remains prime.
\end{remark}

Before we prove the Theorem \ref{p inert}, let us recall some well-known formulas for the order of the so called \textit{ambiguous} and \textit{strong ambiguous class groups}. Let $L/F$ be a cyclic extension of number fields, and $\sigma$ a generator of the Galois group $\Gal(L/F)$. We call an ideal class $[c]\in \Cl(L)$ an ambiguous class if $[c]^\sigma=[c]$. We call an ideal class $[c]\in \Cl(L)$ a strongly ambiguous class if $c^\sigma=c$, that is 
\[
\Am(L/F):=\{ [c]\in \Cl(L) \, | \, [c]^\sigma=[c]\}
\]
and
\[
\Am(L/F)_{st}:=\{ [c]\in \Cl(L) \, |\, c^\sigma=c\}.
\]
In other words, the group $\Am(L/F)$ consists of ideal classes that are fixed by the Galois group $\Gal(L/F)$, and the group $\Am_{st}(L/F)$ consists of ideal classes that contain an ideal that is fixed by the Galois group $\Gal(L/F)$. The order of these groups are given by Chevalley \cite{MR3533015}:

\begin{equation} \label{Cheva}
    \begin{split}
    |\Am(L/F)|&=\frac{h_F\prod_{v}e_{v}}{[L:F][\mO_F^*:\mO_F^*\cap \Nrm(L^*)]}\\
|\Am(L/F)_{st}|&=\frac{h_F\prod_{v}e_{v}}{[L:F][\mO_F^*:\Nrm(\mO_L^*)]}    
    \end{split}   
\end{equation}
where the product is taken over all places $v$ of $L$, $e_v$ is the ramification degree of $v$ in $L/F$, and $h_F$ is the class number of $F$. For any abelian group $M$, we write $M[p^\infty] = \{x\in M \,| \, p^nx=0 \text{ for some } n\}$ to denote the $p$-part of $M$.

\begin{proof}[Proof of Theorem \ref{p inert}]
    Proposition \ref{Hn=An} gives the implication $(b)\Longrightarrow (a)$.

  Now assume that  $A_0=H_0$.  Let $\sigma$ be the generator of $\Gal(K_\infty/K)$. Hence $\sigma$ is also the generator of Galois group $\Gal(K_n/K)$ by restriction. Under our assumptions, there are only two primes $\PP$ and $\Tilde{\PP}$ above $p$ in $K$. Recall we assume that the $\Z_p$-extension $K_\infty/K$ is totally ramified over $\PP$ and unramified over $\tilde{\PP}$ by definition. By Proposition \ref{B bounded}, the size of $B_n=\Am(K_n/K)[p^\infty]$ is bounded. Hence the subgroup $\Am(K_n/K)_{st}[p^\infty]$ is bounded. By (\ref{Cheva}), 
  \[
  |\Am(K_n/K)_{st}|=\frac{h_K\prod_{v}e_{v}}{[K_n:K][\mO_K^*: \Nrm(\mO_{K_n}^*)]}=\frac{h_K \cdot p^n}{p^n[\mO_K^*: \Nrm(\mO_{K_n}^*)]}=\frac{h_K}{[\mO_K^*: \Nrm(\mO_{K_n}^*)]}
  \]
 and so it must be that $[\mO_K^*: \Nrm(\mO_{K_n}^*)]$ is bounded.
  
  Suppose $A_n\neq H_n$ for some $n$. Then by Lemma \ref{x sigma}, there is $c\in A_n$ such that $c\not \in H_n$ and $c^{\sigma-1}\in H_n$. Hence there exists $m$ such that $i_{n,m}(c^{\sigma-1})=0$. Let $c'=i_{n,m}(c)$, and let $\alpha$ be an ideal of $K_m$ such that  $[\alpha]\in c'$. Then $\alpha^{\sigma-1}=(\beta)$ for some $\beta\in K_m^*$, and $\Nrm_{m,0}(\beta)=\varepsilon\in \mO_{K}^*$. Since $[\mO_K^*: \Nrm(\mO_{K_n}^*)]$ is bounded, we know $\Nrm_{s,0}(\beta)=\varepsilon^{s-m}\in \Nrm_{s,0}(\mO_{K_s}^*)$ for $s$ sufficiently larger than $m$.  So, there exists $\eta \in \mO_{K_s}^*$ such that $\Nrm_{s,0}(\beta)=\Nrm_{s,0}(\eta)$, and there is $\gamma\in K_s^*$ such that $\beta\eta^{-1}=\gamma^{\sigma-1}$ by Hilbert's Theorem 90.  Therefore,
  \[
  (\alpha\mO_{K_s})^{\sigma-1}=(\beta)=(\beta\eta^{-1})=(\gamma^{\sigma-1}).
  \]
Hence the ideal class $i_{n,s}(c)$ contains a fractional ideal $ \alpha\mO_{K_s}(\gamma)^{-1}$ that is invariant under the action of $\Gal(K_s/K)$. In other words, $i_{n,s}(c)\in \Am_{st}(K_s/K)[p^\infty]$. Notice that  $\Am_{st}(K_s/K)[p^\infty]=i_{0,s}(A_0)D_s$. We know $i_{0,s}(A_0)\subset H_s $ by assumption and $D_s\subset H_s$ by Corollary \ref{D become principal}. Therefore $i_{n,s}(c)\in \Am_{st}(K_s/K)[p^\infty] \subset H_s$, which contradicts our assumption that $c\not\in H_n$. 
\end{proof}

\begin{lemma}\label{x sigma}
    Let $ \sigma$ be a generator of the cyclic group $G=\Z/p^n\Z$. Let $X\neq \{0\}$ be an abelian $p$-group with an action of $\Z/p^n\Z$ on it. Then there is an element $x\in X$ such that $x\neq 0$ and $x^{\sigma-1}=0$.
\end{lemma}
\begin{proof}
    Consider the following exact sequence,
    \[
    0\rightarrow X^G\rightarrow X\xrightarrow{\sigma-1}X\rightarrow X/X^{\sigma-1}\rightarrow0
    \]
    As usual we write $X^G$ to be the set of elements of $X$ fixed under the action of $G$. Since $G$ and $X$ are $p$-groups, we have that $X^G$, and consequently $X/X^{\sigma-1}$, are nontrivial. Next, consider the exact sequence 
    \[
     0\rightarrow (X^{\sigma-1})^G\rightarrow X^{\sigma-1}\xrightarrow{\sigma-1}X^{\sigma-1}\rightarrow X^{\sigma-1}/X^{(\sigma-1)^2}\rightarrow0.
    \]
Continuing with the same analysis, we have a filtration,
\[
X\supsetneq X^{\sigma-1}\supsetneq X^{(\sigma-1)^2}\supsetneq \cdots\supsetneq X^{(\sigma-1)^k}=0
\]
for some integer $k$. Taking a nontrivial element $x\in X^{(\sigma-1)^{k-1}}$, we have $x^{\sigma-1}=0$.
\end{proof}

\section{The Case Where $p$ Splits Completely in $K^+$}\label{sec3}
In this section, we assume that $p$ splits completely in $K^+$. We once again keep the assumptions \eqref{ass1}, \eqref{ass3}, and \eqref{ass2} as in the previous section. Goto \cite{MR2293501} studies this case for an abelian CM field $K$, but here we do not need to assume $K$ is abelian.

\begin{theorem}\label{TH4.1}
    Assume that $p$ splits completely in $K^+$ and Leopoldt's conjecture holds for $K$. Consider the $S$-ramified $\Z_p$-extension $K_\infty/K$ defined by theorem \ref{unique Zp}.  The following two statements are equivalent:
    \begin{enumerate}
        \item $B_n=D_n$ for all sufficiently large $n$.
        \item $|A_n|$ is bounded as $n\rightarrow \infty$.
    \end{enumerate}
\end{theorem}
\begin{remark}
      Greenberg \cite[Theorem 2]{MR0401702} states a similar criterion for the cyclotomic $\Z_p$-extension of a totally real field when $p$ splits completely. The method of proof is also similar to Greenberg's. 
\end{remark}
\begin{proof}
    Assume that $B_n=D_n$ for all sufficiently large $n$. Since $\Nrm_{m,n}:D_m\rightarrow D_n$ is surjective and $B_n$ is bounded by Proposition \ref{B bounded}, we know that $\Nrm_{m,n}: B_m\rightarrow B_n$ is isomorphism for all $m\geq n\geq n_0$ for some $n_0$. Let $\Ker(\Nrm_{m,n})$ be the kernel of the map $\Nrm_{m,n}: A_m\rightarrow A_n$. Then $\Ker(\Nrm_{m,n})\cap B_m=1 $ for all  $m\geq n\geq n_0$ for some $n_0$.  View $\Ker(\Nrm_{m,n})$ as an abelian $p$-group with the usual action of $\Gal(K_m/K)$. By the general theory of group actions, the fixed point of a nontrivial abelian $p$-group by a $p$-group is nontrivial. The fixed point of  $\Ker(\Nrm_{m,n})$ by $\Gal(K_m/K)$ is $\Ker(\Nrm_{m,n})\cap B_m=1 $. Hence $\Ker(\Nrm_{m,n})=1$, and thus $\Nrm_{m,n}: A_m\rightarrow A_n$ is an isomorphism when $m\geq n\geq n_0$. This implies that  $|A_n|$ is bounded. 

    Now assume that $|A_n|$ is bounded as $n\rightarrow \infty$. We will prove that 
    \[
    B_n=\Am(K_n/K)[p^\infty]=\Am_{st}(K_n/K)[p^\infty]
    \]
    when $n$ is large enough. Recall that $\Am_{st}(K_n/K)[p^\infty]=i_{0,n}(A_0)D_n$ and $i_{0,n}(A_0)$ will become trivial when $n$ is large enough by Proposition \ref{Hn=An}. Therefore, $B_n = \Am_{st}(K_n/K)[p^\infty]$ implies $B_n=D_n$ when $n$ is large enough. 

    Since $|A_n|$ is bounded, reverse the argument in the first paragraph to get $\Nrm_{m,n}:B_m\rightarrow B_n$ is an isomorphism when $m\geq n\geq n_0$ for some $n_0$. Let $c\in B_n$ and take $c'\in B_m $ such that $\Nrm_{m,n}(c')=c$. Let $J $ be an ideal of $K_m$ such that $[J]=c'$ and let $I=\Nrm_{m,n}(J)$. Then $[I]=c$. Let $J^{\sigma-1}=(\beta)$ and $I^{\sigma-1}=(\alpha)$, where $\beta\in K_m^*$ and $\alpha=\Nrm_{m,n}(\beta)$.  Set $\varepsilon=\Nrm_{m,0}(\beta)=\Nrm_{n,0}(\alpha)$. Then $\varepsilon\in \mO_K^*$. 
    
    Let $K_{n,\PP_i}$ be the localization of $K_n$ at $\PP_i$. We have $K_{\PP_i}\cong\Q_p$ because $p$ splits completely in $K$. By local class field theory, a local unit in $\mO_{K_{\PP_i}}$ sits in $\Nrm_{m,0}(K_{m,\PP_i}^*)$ if and only if it is a $p^m$-th power in $K_{\PP_i}^*$. Hence, $\varepsilon$ is a $p^m$-th power in  $K_{\PP_i}^*$ (see Section \ref{sec4}). Let $\mP_i$ be the prime ideal in $K^+$ below $\PP_i$. Since $p$ splits completely in $K$, we have $ K_{\PP_i}\cong K^+_{\mP_i}\cong K_{\tilde{\PP}_i}\cong \Q_p$. 

    Recall that the group generated by the roots of unity of $K$ and $\mO^*_{K^+}$ has index 1 or 2 inside $\mO_{K}^*$ (see Theorem 4.12 in Washington \cite{MR1421575}). Assume that $\varepsilon^2=\varepsilon'\varepsilon''$ such that $\varepsilon'$ is a root of unity inside $K$ and $\varepsilon''$ is a unit of $K^+$.  Since we assume $p$ splits completely in $K$, and $p$ splits completely in $\Q(\zeta_n)$ if and only if $p\equiv 1\Mod{n}$, we have the order of $\varepsilon'$ divides $p-1$. Thus $\varepsilon'$ is a $p^m$-th power.
    
    Since $\varepsilon''$ is a unit in $K^+$, we have $\varepsilon^2$ is a $p^m$-th power in  $K_{\PP_i}^*$ if and only if $\varepsilon''$ is a $p^m$-th power in  $K_{\PP_i}^*$ if and only if $\varepsilon''$ is a $p^m$th power in  $(K^+)_{\mP_i}^*$ if and only if
    $\varepsilon''$ is a $p^m$-th power in  $K_{\tilde{\PP}_i}^*$if and only if
    $\varepsilon^2$ is a $p^m$-th power in  $K_{\tilde{\PP}_i}^*$. In other words, $\varepsilon^2$ is a $p^m$-th power after localization at any prime above $p$. Let $x,y \in \Z$ such that $2x+p^my=1$, so that $\varepsilon=\varepsilon^{2x+p^my}$ is a $p^m$-th power after localization at any prime above $p$. By Leopoldt's conjecture and enlarging $m$ if necessary, we may assume $\varepsilon=\eta^{p^n}$ for some $\eta \in \mO_{K}^*$.  

    Since $\Nrm_{n,0}(\alpha \eta^{-1})=1$, there exists $\gamma\in K_n$ such that $\alpha \eta^{-1}=\gamma^{\sigma-1}$. Thus
    \[
   I^{\sigma-1}=(\alpha)=(\alpha \eta^{-1})=(\gamma^{\sigma-1})
    \]
    
So the ideal class $c\in B_n$ contains a fractional ideal $I(\gamma)^{-1}$ that is fixed by $\Gal(K_n/K)$. Hence $c\in \Am_{st}(K_n/K)[p^\infty]$. Therefore, $B_n=\Am_{st}(K_n/K)[p^\infty]$.
    
\end{proof}

\section{The Ambiguous Class Groups}\label{sec4}

Let $p$ be an odd prime, and $K$ a CM field with maximal real subfield $F$ such that assumptions \ref{ass1} and \ref{ass4} hold.  In this section, we compare the $S$-ramified $\Z_p$-extension $K_\infty/K$ of Theorem \ref{unique Zp} and the cyclotomic $\Z_p$-extension $F^c_\infty/F$ for the totally real field $F$ by computing a certain norm index. Fukuda and Komatsu \cite{MR0845905} give a numerical criterion to determine if $\lambda=0$ for the cyclotomic $\Z_p$-extension of a real quadratic field. We will give an analogous result for the $S$-ramified $\Z_p$-extension $K_\infty/K$ of an imaginary biquadratic field.  First, we will review some needed results from local class field theory (see \cite{milneCFT} for more details).

\begin{theorem}[Local Artin Reciprocity]\label{LAM}
    Let $K/F$ be an Abelian Galois extension local fields.  Then the local Artin map gives an isomorphism
    \[
    F^{*}/\Nrm_{K/F}(K^*) \cong \Gal(K/F).
    \]
\end{theorem}

\begin{theorem}[The Local to Global Principal]\label{L2G}
Suppose $L/K$ is a cyclic Galois extension.  Then if $\gamma \in K$ is a local norm from $L_v$ for all places $v$ of $L$, then $\gamma$ is a global norm from $L$.
\end{theorem}

Now, suppose that $p$ splits completely in $K/\Q$, and that $K_n$ is the $n$-th layer in a $\Z_p$-extension of $K$.  Then $K_n/K$ is cyclic and is unramified outside of the primes above $p$. If $v$ is a place of $K_n$ unramified in $K_n/K$, then the norm map of the resulting local fields is surjective at the group of local units.  Suppose $\PP$ is a prime of $K$ above $p$ which is totally ramified in $K_n$, and $\PP_n$ is a prime of $K_n$ above $\PP$. Let $(K_n)_{\PP_n}$ be the completion of $K_n$ at the  prime $\PP_n$ and $K_0=K$.  By Theorem \ref{LAM},
\[
K_{\PP}^{*}/\Nrm_n((K_n)_{\PP_n}^*) \cong \Z/p^n\Z
\]
where $\Nrm_n$ is the norm map from $(K_n)_{\PP_n}$ to $K_{\PP}$. Since we assume that $p$ splits completely in $K$, we have $K_{\PP}\cong \Q_p$. Let $\pi$ and $\varpi$ be uniformizers of $K_{\PP}$ and $(K_n)_{\PP_n}$ respectively, such that $\Nrm_n(\varpi) = \pi$. Let  $\mO_{\pi}$ and $\mO_{\varpi}$ be the local rings of integers in each field. Let $\mu_{p-1}$ be the group of $(p-1)$-st roots of unity. We have 
\[
(K_n)_{\PP_n}^{*} \cong \varpi^{\Z} \times \mu_{p-1} \times (1 + \varpi\mO_\varpi)
\]
and 
\[
K_{\PP}^{*} \cong  \pi^{\Z} \times\mO_\pi^*
 \cong  \pi^{\Z} \times \mu_{p-1} \times (1 + \pi\mO_\pi)
\]
Let $\psi:K_{\PP}^{*} \rightarrow \mO_\pi^*$ be the projection of onto the factor $\mO_\pi^*$. Now, $\Nrm_n((K_n)_{\PP_n}^{*}) \cong \pi^{\Z} \times \mu_{p-1} \times (1 + \pi^m \mO_\pi)$ for some $m$, and Theorem \ref{LAM} implies $m = n+1$.  Let us now prove the following useful lemma.

\begin{lemma}\label{LCFT}
Assume $p$ splits completely in $K$ and that the ramified primes in the $\Z_p$-extension of $K_\infty/K$ are totally ramified. Keep the same notation as above and let $\gamma \in \mO_K$.  Then $\gamma \in \Nrm_n(K_n^{*})$ if and only if $\psi(\gamma)$ is a $p^n$-th power modulo $\PP$ for all primes $\PP$ of $K$ above $p$ which are totally ramified in $K_n/K$. In particular, if $\gamma 
\in \mO_K^*$, then  $\gamma \in \Nrm_n(K_n^{*})$ if and only if \[\gamma^{p-1}\equiv 1\mod{\PP^{n+1}}\]
for all $\PP$ ramified in $K_n/K$.
\end{lemma}

\begin{proof}
  Suppose that $\gamma \in \mO_K$ and let $S$ be the set of primes that ramify in $K_n/K$. By the local to global principal \ref{L2G}, $\gamma\in \Nrm_n(K_n^*)$ if and only if $\gamma$ is local norm for all primes $\PP\in S$. Since we are assuming that $K_n/K$ is totally ramified at $\PP\in S$, we have $\Nrm_n((K_n)_{\PP_n})=\pi^\Z\times \mu_p\times (1+\pi^{n+1}\mO_\pi)$. Hence $\gamma\in \Nrm_n(K_n^*)$ if and only if $\psi(\gamma)^{p-1}\in 1+\pi^{n+1}\mO_\pi$ if and only if $\psi(\gamma)^{p-1}\equiv 1 \mod \PP^{n+1}$. 
  Since we assume that $p$ split in $K$, $K_\PP\cong \Q_p$. Hence, $\psi(\gamma)^{p-1} \in 1+\pi^{n+1}\mO_\pi$ if and only if $\psi(\gamma)$ is a $p^n$-th power. The lemma now follows. 
\end{proof}

For a number field $L$ we denote $E(L)$ to be the group of units of $\mO_L$ and $W(L)$ to be the roots of unity in $L$.  Note that if an odd prime $p$ splits completely in $L$ then the order of $W(L)$ is coprime to $p$.  Indeed, if $W(L)$ contained a primitive $p$-th root of unity, then $p$ would be ramified in $L/\Q$.

Let $p$ be an odd prime, and $K$ a CM field satisfying Leopoldt's conjecture with $F$ its maximal totally real subfield.  Further, suppose that $p$ splits completely in $K/\Q$.  We define $S$ and $S^+$ as we did in the previous sections, for example, see the beginning of section \ref{sec1}.  Let $K_{\infty}/K$ be the $S$-ramified $\Z_p$-extension $K \subseteq K_1 \subseteq \dots \subseteq K_{\infty}$, and let $F \subseteq F_1^c \subseteq \dots \subseteq F_{\infty}^c$ be the cyclotomic $\Z_p$-extension of $F$.  We assume that any ramified primes in $K_{\infty}/K$ or $F_{\infty}^c/F$ are totally ramified. We still keep assumption s\eqref{ass1}, \eqref{ass3}, \eqref{ass2}, as in previous sections.

\begin{proposition}\label{P5.5}
With the above set up,
\[
 E(K)/( \Nrm_n(K_n^*) \cap E(K))  \cong  E(F) /(\Nrm_n((F^c_n)^*) \cap E(F) )
 \]
\end{proposition}

\begin{proof}
For convenience we write $H_n(K) = \Nrm_n(K_n^*) \cap E(K)$ and $H_n(F) = \Nrm_n((F^c_n)^*) \cap E(F)$. Consider the map $\Theta: E(F) \to  E(K)/H_n(K)$, which is the inclusion $E(F) \hookrightarrow E(K)$ followed by the quotient map $E(K) \to E(K)/H_n(K)$.  Notice that $W(K) \subseteq H_n(K)$, since the order of $W(K)$ is coprime to $p$, and $E(K)/H_n(K)$ is a $p$-group.  Thus, we have the containment $E(F)W(K) \subseteq E(F)H_n(K) \subseteq E(K)$.
Now, by Theorem 4.12 of Washington \cite{MR1421575} we have that $[E(K) : E(F)W(K)] \leq 2$, so it must be that $[E(K) : E(F)H_n(K)] \leq 2$.  But $H_n(K) \subseteq E(F)H_n(K) \subseteq E(K)$, and $E(K)/H_n(K)$ is a $p$-group.  This forces $E(F)H_n(K) = E(K)$ (recall that we are assuming $p$ is odd).  The image of $E(F)$ under $\Theta$ is $E(F)H_n(K)/H_n(K)$, so the above argument shows that $\Theta$ is surjective.

Now suppose that $\beta \in \Ker\Theta$.  Then we have that $\beta \in H_n(K)$, and so by Lemma \ref{LCFT}, for any prime $\PP$ of $K$ above $\p$ we have
\[
\beta^{p-1} \equiv 1 \Mod{\PP^{n+1}}.
\]
Let $\mP = \PP \cap \mO_F$ and notice $\beta \in \mO_F$.  This implies that $\beta^{p-1} \equiv 1 \Mod{\mP^{n+1}}$ for all primes $\mP$ of $F$ above $p$. Hence $\beta \in H_n(F)$, and thus $\Ker\Theta \subseteq H_n(F)$.  Suppose that $\beta \in H_n(F)$.  Then By Lemma \ref{LCFT}, we have that 
\[
\beta^{p-1} \equiv 1 \Mod{\mP^{n+1}}
\]
for all primes $\mP$ of $F$ above $p$.  For any prime $\mP$ of $F$ above $p$, we have $\mP\mO_K = \PP \bar{\PP}$, and
\[
\beta^{p-1} \equiv 1 \Mod{\PP^{n+1}} \tab \text{ and } \beta^{p-1} \equiv 1 \Mod{\bar{\PP}^{n+1}}
\]
since $\beta \in E(F)$.  Therefore, $\beta \in H_n(K)$ by Lemma \ref{LCFT}.  This shows that $\Ker\Theta = H_n(F)$ so that $E(F)/H_n(F) \cong E(K)/H_n(K)$.
\end{proof}

\begin{remark}
   Notice that $F_n^c\not \subset K_n$. Proposition \ref{P5.5} is interesting because we do not have a direct relation between $K_n$ and $F^c_n$. Though the extensions $K_n/K$ and $F^c_n/F$ are globally different and unrelated, they are locally  similar due to the assumption that $p$ splits completly in $K/\Q$.  
\end{remark}

Given an extension $L/M$ of number fields, let $\Am_p(L/M)$ denote the $p$-ambiguous class group, that is
\[
\Am_p(L/M) = A(L)^{\Gal(L/M)}
\]
where $A(L)$ is the $p$-class group of $L$.  

\begin{corollary} \label{ambiguous com}
Assume that $p$ splits completely in $K$ and ramified primes in $K_{\infty}/K$ or $F_{\infty}^c/F$ are totally ramified. With the above setup, we have
\[
\frac{|\Am_p(K_n/K)|}{|A(K)|} = \frac{|\Am_p(F^c_n/F)|}{|A(F)|}.
\]
\end{corollary}

\begin{proof}
Chevalley's formula (\ref{Cheva}) has that
\[
|\Am_p(K_n/K)| = |A(K)| \frac{\prod_{ \PP} e(\PP_n/\PP)}{[K_n:K] [E(K) : \Nrm_n(K_n^*) \cap E(K) ]}
\]
where the product ranges over primes $\PP$ of $K$ above $p$, and $e(\PP_n/\PP)$ denotes the ramification index.  If $\PP\in S$, then $e(\PP_n/\PP) = p^n$.  Similarly,
\[
|\Am_p(F^c_n/F)| = |A(F)| \frac{\prod_{ \mP} e(\mP_n/\mP)}{[F^c_n:F] [E(F) : \Nrm_n((F^c_n)^*) \cap E(F) ]}.
\]
Now, if $\mP$ ramifies in $F_n^c$, then $e(\mP_n/\mP) = p^n$, and the number of ramified primes in $F_n^c/F$ is the same as the number of ramified primes in $K_n/K$.  This together with the previous proposition proves the Corollary.  
\end{proof}

\subsection{The $S$-Ramified $\Z_p$-Extensions of Imaginary Biquadratic Fields}

As an application, we prove results analogous to those of Fukuda Komatsu \cite{MR0845905} for the $S$-ramified extension of an imaginary biquadratic field.

Let $m,d \in \Z^+$ that are squarefree and coprime.  Denote $k = \Q(\sqrt{-m})$, $F = \Q(\sqrt{d})$, $K = Fk$, and $\vep$ to be the fundamental unit for $K$.  Suppose that $p > 2$ is a prime that splits completely in $K$, with $p\mO_k = \p\tilde{\p}$ and $\p\mO_{K} = \PP\bPP$.  Suppose that 
\[
k \subseteq k_1 \subseteq k_2 \subseteq \dots \subset \bigcup_n k_n = k_{\infty}
\]
 is the unique $\Z_p$-extension of $k$ unramified outside $\p$. Put $K_n=Fk_n$ and $K_\infty=Fk_\infty$. Then $K_\infty/K$ is the $S$-ramified $\Z_p$-extension of $K$, where $S$ is the set of primes above $\p$ in $K$. 
 
 Denote $E_n$ to be the units of $\mO_{K_n}$, and $\PP_n$ to be the prime of $\mO_{K_n}$ that lies above $\PP$.  Let $\Nrm_{n,m}: K_n \to K_m$ be the norm map, and $\Nrm_n : K_n \to K$ the norm map from $K_n$ to $K$. For any number field $L$, we denote $h_L$ to be the class number of $L$. 

 The following is an analogue of the main theorem in \cite{MR0845905}:

 \begin{theorem}\label{TH5.7}
 Suppose $p \nmid h_K$, and that $r$ is the smallest positive integer such that 
 \[
 \vep^{p-1}\equiv 1 \Mod{\bPP^r}.
 \]
 If $\Nrm_{r-1}(E_{r-1}) = E_0$ then $\mu=\lambda = 0$ for the $S$-ramified $\Z_p$-extension of $K_\infty/K$ defined above.
 \end{theorem}

\begin{remark}
  \cite[Proposition 1]{MR0190132}  tells us that $p \nmid h_K$ implies $p \nmid h_F$ and $p \nmid h_k$. Hence primes ramified in the $\Z_p$-extension $k_\infty/k$ are totally ramified. Thus the primes ramified in the $\Z_p$-extension $K_\infty/K$ is totally ramified. Since $[F:\Q]=2$ and $p\geq 3$, primes ramified in the cyclotomic $\Z_p$-extension $F_\infty^c/F$ are also totally ramified. Therefore, we may apply Proposition \ref{P5.5} in the following proof.  

\end{remark}
 \begin{proof}
Suppose that $\Nrm_{r-1}(E_{r-1}) = E_0$.  Then there is $\beta \in E_{r-1}$ such that $\Nrm_{r-1}(\beta) = \vep$, and so for any $n \geq r-1$ we have $\vep^{p^{n-r+1}} \in \Nrm_n(E_n)$.  Thus, $|E_0/\Nrm_n(E_n)| \leq p^{n-r+1}$.  Now, $|E_0/(\Nrm_n(K_n^*) \cap E_0)| = |E(F)/(\Nrm_n((F_n^c)^*) \cap E(F))|$ by Proposition \ref{P5.5}.  Meanwhile, Fukuda and Komatsu \cite[Lemma 2]{MR0845905} calculate that $|E(F)/(\Nrm_n((F_n^c)^*) \cap E(F))| = p^{n-r+1}$, hence $|E_0/\Nrm_n(E_n)| \leq p^{n-r+1} = |E_0/(\Nrm_n(K_n^*) \cap E_0)| \leq |E_0/\Nrm_n(E_n)|$.
By Chevalley's formula, $B_n = D_n$ for all $n \geq r-1$ so that $\mu=\lambda = 0$ by Theorem \ref{TH4.1}.
 \end{proof}

 \begin{lemma}\label{L5.9}
    Suppose that $p$ splits in $K$. Let $\p$ be a prime above $p$ in $F$. Let $\PP$ be a prime above $\p$ in $K$, $\vep_K$ the fundamental unit of $E(K)$, and $\vep_F$ the fundamental unit of $F$.  Then
     \[
     \vep_K^{p-1} \equiv 1 \Mod{\PP^n} \tab \iff \tab \vep_F^{p-1} \equiv 1 \Mod{\p^n}.
     \]
 \end{lemma}

\begin{proof}
  Let $W(K)$ be the group of roots of unity in $K$. Since $p$ splits completely in $K$,  we have $\#W(K)$ divides $p-1$.  Recall that $[E(K):W(K)E(F)]\leq 2$ by \cite[Theorem 2.13]{MR1421575}, so $\varepsilon_K^2=\varepsilon_F^q\zeta$ for some $q\in \Z$ and $\zeta\in W(K)$. Since $E(F)\subset E(K)$, we have $\varepsilon_F=\varepsilon_K^r\eta$ for some $r\in \Z$ and $\eta\in W(K)$.  Let $F_\p$ be the completion of $F$ at $\p$ and $K_\PP$ be the completion of $K$ at $\PP$. Since $p$ splits completely in $K$, we have $F_\p\cong K_\PP\cong \Q_p$. Then $\varepsilon_F^{p-1}\equiv 1 \Mod {\p^n}$ if and only if $\varepsilon_F^{p-1}$ is a $p^{n-1}$-th power if and only if $\varepsilon_K^{p-1}$ is a $p^{n-1}$-th power if and only if $  \vep_K^{p-1} \equiv 1 \Mod{\PP^n}$.
\end{proof}
\begin{corollary}[Analogue of Lemma 4 in \cite{MR0845905}]
Let $\vep$ be a fundamental unit of $E(K)$.  Suppose that $p \nmid h_K$, and that
\[
\vep^{p-1} \equiv 1 \Mod{\bPP^2} \tab \text{ but } \tab \vep^{p-1} \not\equiv 1 \Mod{\bPP^3}.
\] 
Write $\PP^{h_K} = (\alpha)$, and suppose that 
\[
\alpha^{p-1} \equiv 1 \Mod{\bPP} \tab \text{ but } \tab \alpha^{p-1} \not\equiv 1 \Mod{\bPP^2}
\]
Then $\mu=\lambda = 0$.
\end{corollary}

\begin{proof}
We again follow the proof of Fukuda and Komatsu in \cite{MR0845905}.  Under our assumptions, Lemma \ref{LCFT} implies that $E(K) = \Nrm_1(K_1^*) \cap E(K)$.  Therefore,
\[
[B_1 : D_1] = [\Nrm_1(K_1^*) \cap E(K): \Nrm_1(E_1)] = [E(K) : \Nrm_1(E_1)].
\]
Therefore, by Theorem \ref{TH5.7}, if $B_1 = D_1$ then $\lambda = 0$.  By Lemma \ref{L5.9}, we have that 
\[
\vep_F^{p-1} \equiv 1 \Mod{\p^2}\tab \text{but}\tab \vep_F^{p-1} \not\equiv 1 \Mod{\p^3}.
\]
So by Proposition 1 in \cite{MR0845905} combined with Corollary \ref{ambiguous com}, we have $|B_1| = p$.  Let $\PP_1$ be the prime of $K_1$ above $\PP$ in $K$.  Then the class of $\PP_1^{h_K}$ generates $D_1$ (here we are using the assumption that $p \nmid h_K$).  We will show that $\PP_1^{h_K}$ is not principle.  Indeed, suppose that $\PP_1^{h_K} = (\alpha_1)$.  Then 
\[
\Nrm_1(\alpha_1) = \alpha \vep^t
\]
for some $t \in \Z$.  Now, Lemma \ref{LCFT} implies that $\Nrm_1(\alpha_1)^{p-1}\equiv 1 \Mod{\bPP^2} $. Thus $\alpha^{p-1} \equiv 1 \Mod{\bPP^2}$, which contradicts our assumptions on $\vep$ and $\alpha$.
\end{proof}

\printbibliography
\end{document}